\documentclass[12pt,a4paper]{amsart}
\usepackage{amsmath,amssymb,amsfonts,amsthm,amsopn,dsfont}
\numberwithin{equation}{section}
\newtheorem{theorem}{Theorem}[section]

\newtheorem{proposition}[theorem]{Proposition}
\newtheorem{definition}[theorem]{Definition}

\newtheorem{remark}[theorem]{Remark}

\def\a{\alpha}

\def\d{\delta}

\def\d{{\rm d}'}
\newcommand{\R}{\mathbb{R}}

\newcommand{\rn}{\mathbb{R}^{n}}
\newcommand{\cD}{\mathcal{D}}
\newcommand{\C}{\mathbb{C}}

\newcommand{\zp}{z^\prime}
\newcommand{\zpp}{z^{\prime\prime}}

\newcommand{\tp}{t^\prime}
\newcommand{\tpp}{t^{\prime\prime}}

\newcommand{\ozp}{\overline{z}^\prime}
\newcommand{\ozpp}{\overline{z}^{\prime\prime}}

\begin{document}

\title[]{Gevrey local solvability in locally integrable structures}\author{Francesco Malaspina and
Fabio Nicola}
\address{Dipartimento di Scienze Matematiche,
Politecnico di Torino, corso
Duca degli Abruzzi 24, 10129
Torino, Italy}
\address{Dipartimento di Scienze Matematiche,
Politecnico di Torino, corso
Duca degli Abruzzi 24, 10129
Torino, Italy}
\email{francesco.malaspina@polito.it}
\email{fabio.nicola@polito.it}
\thanks{}
\begin{abstract}
We consider a locally integrable real-analytic structure,
and we investigate the local solvability in the category of
Gevrey functions and ultradistributions of the complex $\d$
naturally induced by the de Rham complex. We prove that the
so-called condition $Y(q)$ on the signature of the Levi
form, for local solvability of $\d u=f$, is still necessary
even if we take $f$ in the classes of Gevrey functions and
look for solutions $u$ in the corresponding spaces of
ultradistributions.
\end{abstract}
\maketitle

\section{Introduction and statement of
the results}\label{intro}

Consider a real-analytic
manifold $M$ of dimension
$m+n$. A real-analyitc
locally integrable structure
on $M$, of rank $n$, is
defined by a real analytic
subbundle
$\mathcal{V}\subset\C TM$ of
rank $n$, satisfying the
Frobenius condition and such
that the subbundle
$T^\prime\subset\C T^\ast M$
orthogonal to $\mathcal{V}$
is locally spanned by exact
differentials. As usual we
will denote by
$T^0=T^\prime\cap T^\ast M$
the so-called {\it
characteristic set}. For any
 open subset
$\Omega\subset M$ and $s>1$
the space $G^s(\Omega,
\Lambda^{p,q})$ of
$(p,q)$-forms with Gevrey
coefficients of order $s$
 is then
defined (see Section 3 below
and Treves \cite{t2}) and the
de Rham differential induces
a map
\[
{\rm d}^\prime: G^s(\Omega,
\Lambda^{p,q})\to G^s(\Omega,
\Lambda^{p,q+1}).
\]
Similarly, the de Rham
differential induces a
complex on the space of
``ultra-currents"
$\cD'_s(\Omega,
\Lambda^{p,q})$, i.e. forms
with ultradistribution
coefficients:
\[
{\rm d}^\prime:
\mathcal{D}'_s(\Omega,
\Lambda^{p,q})\to
\cD'_s(\Omega,
\Lambda^{p,q+1}).\] When
$\mathcal{V}\cap\overline{\mathcal{V}}=0$
the structure is called $CR$
and ${\rm d}^\prime$ is the
so-called tangential
Cauchy-Riemann operator.\par
We are interested in
necessary conditions for the
Gevrey local solvability
problem for the complex ${\rm
d}'$ to hold near a given
point $x_0$.
\begin{definition}\label{risolubilita} We say
that the complex ${\rm d}'$
is locally solvable near
$x_0$ and in degree $q$, $1\leq q\leq n$, in
the sense of
ultradistribution of order
$s$, if for every sufficiently small open
neighborhood $\Omega$ of
$x_0$ and every
 cocycle $f\in
G^s(\Omega,\Lambda^{0,q})$ there
exists an open neighbourhood
$V\subset U$ of $x_0$ and a
ultradistribution section
$u\in
\mathcal{D}^\prime_s(V,\Lambda^{0,q-1})$
solving ${\rm d}^\prime u=f$
in $V$.
\end{definition}
The analogous problem in the
setting of smooth functions and Schwartz
distributions has been
extensively considered, see
e.g.
\cite{an2,an1,libro,cor0,cor01,cor1,cor3,cor2,mi,nac,nac2,nicola,peloso2,t2}, inspired by the results in \cite{ho0,lewy} for scalar operators of principal type; see also \cite{lerner,rodino2} as general references for the problem of local solvability of scalar linear partial differential operators.  \par
Several geometric invariants were there
introduced, e.g.\ the signature of the Levi form recalled below, which represent obstructions to
the solvability in the sense
of distributions, that is,
for some smooth $f\in
C^\infty(U,\Lambda^{0,q})$
there is no distribution
solution $u\in
\cD'(V,\Lambda^{0,q-1})$ to
${\rm d}'u=f$ in $V$, for
every neighbourhood $V\subset
U$ of $x_0$.
\par It is therefore natural
to wonder whether, under the same condition as in the smooth category, ${\rm d}'$
is still  non-solvable even if we
choose $f$ in the smaller
class of Gevrey functions
$G^s(U,\Lambda^{0,q})\subset
C^\infty(U,\Lambda^{0,q})$
and we look for solutions in the
larger class of ultradistributions
$\cD'_s(V,\Lambda^{0,q})\supset
\cD'(V,\Lambda^{0,q})$, as in
Definition
\ref{risolubilita}. In this
note we present a
result in this
direction.\par Let us note that general sufficient conditions for local solvability in the Gevrey category have been  recently  obtained in \cite{caetano2}; see also \cite{caetano,mi2}.\par
We recall
 that at any point $(x_0,\omega_0)\in T^0$ it is
  well defined a sesquilinear
  form $\mathcal{B}_{(x_0,\omega_0)}:\mathcal{V}_{x_0}
  \times\mathcal{V}_{x_0}\to\mathbb{C}$,
  ($\mathcal{V}_{x_0}$ is the fibre above $x_0$) by
\[
\mathcal{B}_{(x_0,\omega_0)}(\mathbf{v}_1,\mathbf{v}_2)=\left\langle\omega_0,(2\iota)^{-1}[V_1,\overline{V_2}]|_{x_0}\right\rangle,
\]
with $\mathbf{v}_1,\mathbf{v}_2\in\mathcal{V}_{x_0}$,
where $V_1$ and $V_2$ are smooth sections
 of $\mathcal{V}$ such that $V_1|_{x_0}=\mathbf{v}_1$,
  $V_2|_{x_0}=\mathbf{v}_2$.
  The associated quadratic form $\mathcal{V}_{x_0}
  \ni\mathbf{v}\mapsto \mathcal{B}_{(x_0,\omega_0)}
  (\mathbf{v},\mathbf{v})$, or
  $\mathcal{B}_{(x_0,\omega_0)}$ itself,
  is known as {\it Levi
  form}.\par Here is our
  result.
\begin{theorem}\label{teor2} Let $(x_0,\omega_0)\in T^0$, $\omega_0\not=0$.
 Suppose that $\mathcal{B}_{(x_0,\omega_0)}$ has
 exactly $q$ positive
 eigenvalues, $1\leq q\leq n$, and
 $n-q$ negative eigenvalues,
 and that its restriction to $\mathcal{V}_{x_0}\cap
 \overline{\mathcal{V}}_{x_0}$ is non-degenerate. \par
Then, for every $s>1$, ${\rm
d}'$ is not locally solvable
in the sense of
ultradistributions of order
$s$, near $x_0$ and in degree
$q$.
\end{theorem}
This result therefore strengthens the analogous one in the category of smooth functions and Schwartz distributions, which was proved in \cite{an2} for $CR$ manifolds and in \cite[Theorem XVIII.3.1]{t2} for general locally integrable structure; see also \cite{hnp,nicola2} for partial results when the Levi form is degenerate. As general reference for related results about scalar operators on Gevrey spaces see \cite{rodino}.\par


\section{Preliminaries}\label{sec2}
\subsection{Gevrey functions and ultradistributions}
Let us briefly recall the definition of the classes of Gevrey functions and corresponding ultradistributions; see e.g. \cite[Chapter 1]{rodino} for details.\par
 Let $s>1$ be a real number and $\Omega$ be an open subset of $\rn$; let $C$ be a positive constant. We denote by $G^s(\Omega,C)$ the space of smooth functions $f$ in $\rn$ such that for every compact $K\subset\Omega$,
  \[
 ||f||_{K,C}:=\sup_{\alpha} C^{-|\alpha|}(\alpha!)^{-s}\sup_{x\in K}|\partial^\alpha f(x)|<\infty.
 \]
 This is a Fr\'echet space endowed with the above seminorms.
 We set
 $G^s(\Omega)$ for the usual Gevrey space of order $s$, i.e. $f\in G^s(\Omega)$ if $f$ is smooth in $\Omega$ and for every compact $K\subset\Omega$ there exists $C>0$ such tha $ ||f||_{K,C}<\infty$. We will also consider the space $G^s_0(K,C)$ of functions in $G^s(\Omega,C)$ supported in the compact $K$; it is a Banach space with the norm  $||u||_{K,C}$. Finally we set \[
 G^s_0(\Omega)=\bigcup\limits_{K\subset\Omega,\ C>0} G^s_0(K,C).\]  \par
 The space of $\mathcal{D}'_s(\Omega)$ of ultradistributions of order $s$ in $\Omega$ is by definition the dual of  $G_0^s(\Omega)$, i.e. an element $u\in \mathcal{D}'_s(\Omega)$ is a linear functional on $G_0^s(\Omega)$ such that for every compact $K\subset\Omega$ and every constant $C>0$ there exists a constant $C'>0$ such that
 \[
 |\langle u,f\rangle|\leq C' \|f\|_{K,C},
 \]
 namely $u\in (G^s_0(K,C))'$ for every $K$, $C$.
 Clearly, $\mathcal{D}'_s(\Omega)$ contains the usual space $\mathcal{D}'(\Omega)$  of Schwartz distributions. \par
 We will need the following estimate for Gevrey seminorms of exponential functions.
 \begin{proposition}\label{prorodino}
 Let $\psi$ be a real-analytic function in a neighborhood $\Omega$ of $0$ in $\mathbb{R}^n$; then for every compact subset $K$ of $\Omega$ and every $C>0$, $s>s'>1$, there exists a constant $C'>0$ such that
 \begin{equation}\label{prorodino1}
 ||{\rm exp}(\iota\rho\psi)||_{K,C}\leq C'{\rm
 exp}(a\rho+\rho^{1/s'})
 \end{equation}
 for every $\rho>0$, where $a=\sup\{-{\rm Im}\,\psi(x):\,x\in K\}$.
 \end{proposition}
 \begin{proof}
 By the Fa\`a di Bruno formula (see e.g. \cite[page 16]{GP}) we have, for $|\alpha|\geq 1$,
\begin{align*}
\partial^\alpha e^{\iota\rho\psi(x)}=\sum_{j=1}^{|\alpha|} \frac{{\rm exp}(\iota\rho\psi)}{j!}\sum_{\gamma_1+\ldots+\gamma_j=\alpha\atop |\gamma_k|\geq1}\frac{\alpha!}{\gamma_1!\ldots\gamma_j!}|\partial^{\gamma_1}(\iota\rho \psi(x))|\ldots |\partial^{\gamma_j} (\iota\rho\psi(x))|.
\end{align*}
By assumption there exists a constant $C_1>0$ such that $|\partial^\gamma \psi(x)|\leq C_1^{|\gamma|}\gamma!$ for $x\in K$, $|\gamma|\geq 1$. Hence for every $\alpha$,
\[
\sup_{x\in K}|\partial^\alpha e^{\iota\rho\psi(x)}|\leq e^{a\rho}\alpha!C_2^{|\alpha|}\sum_{j=0}^{|\alpha|}\frac{\rho^j}{j!}
\]
with $C_2=2^{n+1}C_1$, where we used
$$\sum_{\gamma_1+\ldots+\gamma_j=\alpha\atop |\gamma_k|\geq1} 1\leq \prod_{k=1}^n{\alpha_k+j-1 \choose j-1} \leq 2^{|\a|+n(j-1)}\leq 2^{(n+1)|\a|}.$$
Hence we have
\begin{align*}
||{\rm exp}(\iota\rho\psi)||_{K,C}
&\leq e^{a\rho}(\alpha!)^{1-s}(C_2/C)^{|\alpha|}\sum_{j=0}^{|\alpha|}\frac{\rho^j}{j!}\\
 &\leq e^{a\rho}(|\alpha|!)^{1-s}(C_2C_3/C)^{|\alpha|}\sum_{j=0}^{|\alpha|}\frac{\rho^j}{j!},
 \end{align*}
because $|\alpha|!\leq n^{|\alpha|}\alpha!$. Now, we have
$(|\alpha|!)^{1-s'}(j!)^{s'-1}\leq 1$ and by Stirling formula $(|\alpha|!)^{s'-s}(C_2C_3/C)^{|\alpha|}\leq C'$, so that
\[
||{\rm exp}(\iota\rho\psi)||_{K,C}\leq C'e^{a\rho}\sum_{j=0}^{|\alpha|}\frac{\rho^j}{(j!)^{s'}}\leq C'e^{a\rho}\sum_{j=0}^{|\alpha|}\Big(\frac{\rho^{j/s'}}{j!}\Big)^{s'}\leq C' e^{a\rho+s'\rho^{1/s'}}.
\]
 Since this holds for every $1<s'<s$, we can replace the constant $s'$ in front of $\rho^{1/s'}$ by $1$, possibly for a new constant $C'$ and for a slightly lower value of $s'$. Hence \eqref{prorodino1} is proved.
  \end{proof}
  \subsection{Locally integrable
structures} Consider a real-analytic
manifold $M$ of dimension
$N=m+n$. A real-analyitc
locally integrable structure
on $M$, of rank $n$, is
defined by a real analytic
subbundle
$\mathcal{V}\subset\C TM$ of
rank $n$, satisfying the
Frobenius condition and such
that subbundle
$T^\prime\subset\C T^\ast M$
orthogonal to $\mathcal{V}$
is locally spanned by exact
differentials. As usual we
will denote by
$T^0=T^\prime\cap T^\ast M$
the so-called {\it
characteristic set}. Let $k$ be a positive integer, we denote by $\Lambda^k\C T^\ast M$ the $k$-th exterior power of $\C T^\ast M$.
Let us consider complex exterior algebra
$$\Lambda \C T^\ast M=\oplus_{k=0}^N \Lambda^k\C T^\ast M,$$
for any pair of positive integers $p,q$ we denote by $$T^{\prime p,q}$$ the homogeneous of degree $p+q$ in the ideal generated by the $p$-th exterior power of $T^\prime$, $\Lambda^p T^\prime$. We have the inclusion $$T^{\prime p+1,q-1}\subset T^{\prime p,q}$$ which allows us to define
$$\Lambda^{p,q}=T^{\prime p,q}/T^{\prime p+1,q-1}.$$
If $\phi$ is a smooth section of $T^\prime $ over an open subset $\Omega\subset M$, its exterior derivative $d\phi$ is section of $T^{\prime 1,1}$. In other words $$d T^\prime\subset T^{\prime 1,1}.$$
It follows at once from this that, if $\sigma$ is a smooth section of $T^{\prime p,q}$ over $\Omega$, then $d\sigma$ is a section of $T^{\prime p, q+1}$ i.e.
$$d T^{\prime p,q}\subset T^{\prime p,q+1}.$$

Let $s>1$,
the space $G^s(\Omega,
\Lambda^{p,q})$ of
$(p,q)$-forms with Gevrey
coefficients of order $s$
 is
defined, as well as $G^s(\Omega,C;
\Lambda^{p,q})$, $G_0^s(K,C;
\Lambda^{p,q})$, etc, with notation analogous to the scalar case.  \par The
de Rham differential induces then
a map
\[
{\rm d}^\prime: G^s(\Omega,
\Lambda^{p,q})\to G^s(\Omega,
\Lambda^{p,q+1}).
\]
(see  Treves \cite[Section I.6]{t2} for more details).
Similarly, the de Rham
differential induces a
complex on the space of
``ultra-currents"
$\cD'_s(\Omega,
\Lambda^{p,q})$, i.e. forms
with ultradistribution
coefficients:
\[
{\rm d}^\prime:
\mathcal{D}'_s(\Omega,
\Lambda^{p,q})\to
\cD'_s(\Omega,
\Lambda^{p,q+1}).\]
Namely, consider for simplicity the case when $\Omega$ is orientable (in fact, in the sequel we will work in a local chart). Stokes' theorem implies that
\[
\int_\Omega \d u\wedge v=(-1)^{p+q-1}\int_\Omega u\wedge \d v
\]
if $u\in G^s(\Omega,\Lambda^{p,q})$, $v\in G_0^s(\Omega,\Lambda^{m-p,n-q-1})$, and accordingly we can define
\[
\langle \d u, v\rangle=(-1)^{p+q-1}\langle u, \d v\rangle
\]
if $u\in D^\prime_s(\Omega,\Lambda^{p,q})$, $v\in G_0^s(\Omega,\Lambda^{m-p,n-q-1})$.

\section{Local solvability estimates}
We now show that local solvability implies an a priori-estimate. This is analogous to the estimates of H{\"o}rmander \cite{ho0}, Andreotti, Hill and Nacinovich
\cite{an2}, Treves \cite[Lemma VIII.1.1]{t2}, in the framework of Schwartz distributions. \begin{proposition}\label{prop2.1} Suppose that, for
some $s>1$, the complex ${\rm
d}'$ is locally solvable near
$x_0$ and in degree $q$, in
the sense of
ultradistributions of order
$s$ (see Definition
\ref{risolubilita}). Then for every sufficiently small open neighborhood $\Omega$ of $x_0$,
every $C_1>0$, $0<\epsilon<C_2$ there exist a
compact $K\subset\Omega$, an
open neighbourhood
$\Omega'\subset\subset\Omega$
of $x_0$ and a constant
$C'>0$, such that
\begin{equation}\label{apriori}
\big|\int_\Omega f\wedge v\big|\leq
C'\|f\|_{K,C_1}\|{\rm
d}'v\|_{\overline{\Omega'},C_2},
\end{equation}
for every cocycle $f\in
G^s(\Omega,C_1;\Lambda^{0,q})$
and every $v\in
G^s_0(\overline{\Omega'},C_2-\epsilon;\Lambda^{m,n-q}).$
\end{proposition}
It will follow from the proof that $\|{\rm
d}'v\|_{\overline{\Omega'},C_2}<\infty$ if $v\in
G^s_0(\overline{\Omega'},C_2-\epsilon;\Lambda^{m,n-q}).$
\begin{proof}
Let $V_{j+1}\subset V_j\subset \subset\Omega$, $j=1,2,\ldots$, be a fundamental system of neighborhoods of $x_0$. Fix $C_1>0$, $0<\epsilon<C_2$ and consider the space
\begin{multline*}
F_j=\{(f,u)\in G^s(\Omega,C_1;\Lambda^{0,q})\times  G^s_0(\overline{V_j},C_2;\Lambda^{0,q-1})':\\   \d f=0\ {\rm in}\ \Omega,\ \d u=f \ {\rm in}\  G^s_0(\overline{V_j},C_2-\epsilon;\Lambda^{0,q})'\}.
 \end{multline*}
 The last condition means $\langle \d u,v\rangle=\langle f,v \rangle$ for every $v\in G^s_0(\overline{V_j},C_2-\epsilon;\Lambda^{m,n-q})$, which makes sense by transposition, because differentiation maps $G^s_0(\overline{V_j},C_2-\epsilon)\to G^s_0(\overline{V_j},C_2)$ (see e.g. \cite[Proposition 2.4.8]{rodino}) and multiplication by analytic functions preserves the latter space.\par
 Now, by direct inspection one sees that $F_j$ is a closed subspace of \\ $G^s(\Omega,C_1;\Lambda^{0,q-1})\times  G^s_0(\overline{V_j},C_2;\Lambda^{0,q})'$, therefore Fr\'echet. \par
 Let \[
 \pi_{j}:F_j\to \{f\in G^s(\Omega,C_1;\Lambda^{0,q-1}): \d f=0\}
 \]
  be the canonical projection $(f,u)\mapsto f$. The assumption of local solvability implies that
 \[
 \{f\in G^s(\Omega,C_1;\Lambda^{0,q-1}): \d f=0\}=\cup_{j} \pi_j(F_j).
 \]
 By the Baire theorem, there exists $j_0$ such that $\pi_{j_0}(F_{j_0})$ is of second category. By the open mapping theorem, we see that $\pi_{j_0}$ is onto and open: there exists a compact $K\subset\Omega$ and a constant $C'>0$ such that for every cocycle $f\in G^s(\Omega,C_1;\Lambda^{0,q})$, there exists $u\in G_0^s(\overline{V_{j_0}},C_2;\Lambda^{0,q-1})'$ satisfying $\d u=f$ in $G_0^s(\overline{V_{j_0}},C_2-\epsilon;\Lambda^{0,q})'$ and
 \[
 |u|_{{\overline{V_{j_0}},C_2}}:=\sup\limits_{||v||_{\overline{V_{j_0}},C_2}=1}|\langle u,v\rangle|\leq C' ||f||_{K,C_1}.
 \]
 Consider now the bilinear functional $(f,v)\mapsto \int_\Omega f\wedge v=\langle f,v\rangle$, for $f\in G^s(\Omega,C_1;\Lambda^{0,q})$ cocycle, and $v\in G_0^s(\overline{V_{j_0}},C_2-\epsilon;\Lambda^{m,n-q})$. Given such a  $f$, we take $u$ as before, and we get
 \[
 |\langle f,v\rangle|=|\langle \d u,v\rangle|=|\langle u,\d v\rangle|\leq  |u|_{{\overline{V_{j_0}},C_2}}\|\d v\|_{{\overline{V_{j_0}},C_2}}\leq C' \|f\|_{K,C_1}\|\d v||_{{\overline{V_{j_0}},C_2}}.
 \]
  \end{proof}
\section{Proof of
Theorem \ref{teor2}}
We work in a sufficiently small neighborhood $\Omega$ of
the point $x_0$ (to be chosen later), where local solvability holds.  We also take $x_0$ as the origin of the
coordinates, i.e.\ $x_0=0$. Moreover we make use of the special
coordinates, whose existence is proved in see section I.9 of \cite{t2}.
Namely, let $n={\rm dim}_\C\mathcal{V}_0$, $d={\rm dim}_\R T^0_0$,
 $\nu=n-{\rm dim}_\C(\mathcal{V}_0\cap\overline{\mathcal{V}}_0)$. We have the following result.
\begin{proposition}\label{prr}
Let $(0,\omega_0)\in T^0$, $\omega_0\not=0$,
 and suppose that the restriction of the Levi form $\mathcal{B}_{(0,\omega_0)}$ to $\mathcal{V}_0\cap\overline{\mathcal{V}}_0$ is non-degenerate. There exist real-analytic coordinates $x_j,y_j$,$s_k$ and $t_l$, $j=1,\ldots,\nu$, $k=1,\ldots, d$, $l=1,\ldots,n-\nu$, and smooth real valued and real-analitic functions $\phi_k(x,y,s,t)$, $k=1,\ldots d$, in a neighborhood $\mathcal{O}$ of $0$, satisfying
\begin{equation}\label{zero}
\phi_k|_0=0\ \ {\rm and}\ \  d\phi_k|_0=0,
\end{equation} such that
\[
\begin{cases}
z_j:=x_j+\iota y_j,\ j=1,\ldots,\nu,\\ w_k:=s_k+\iota\phi_k(x,y,s,t),\ k=1,\ldots,d,
\end{cases}
\]
define a system of first integrals for $\mathcal{V}$, i.e.\ their differential span $T^\prime|_\mathcal{O}$.\par Moreover, with respect to the basis \[
\left\{\left.\frac{\partial}{\partial
\overline{z}_j}\right|_0,\left.\frac{\partial}{\partial t_l}\right|_0; j=1,\ldots,\nu,\ l=1,\ldots,n-\nu\right\}\]
 of
$\mathcal{V}_0$ the Levi form $\mathcal{B}_{(0,\omega_0)}$ reads
\begin{equation}\label{ama}
\sum_{j=1}^{p^{\prime\prime}}|\zeta_j|^2-\sum_{j=p^{\prime\prime}+1}^{\nu}|\zeta_j|^2+
\sum_{l=1}^{p^\prime}|\tau_l|^2-\sum_{l=p^\prime+1}^{n-\nu}|\tau_l|^2.
\end{equation}
\end{proposition}
\begin{remark}\rm
In particular
\begin{equation}\label{ama2}
{\rm d}^\prime z_j=0,\ {\rm d}^\prime  w_k=0,\quad
j=1,\ldots,\nu;\ k=1,\ldots,d.
\end{equation}
In these coordinates we have $T_0^0={\rm span}_\R\{ds_k|_0;\ k=1,\ldots,d\}$, so that $\omega_0=\sum_{k=1}^d \sigma_k ds_k|_0$, with $\sigma_k\in\R$. By (I.9.2) of \cite{t2} we have $\mathcal{B}_{(0,\omega_0)}(\mathbf{v}_1,\mathbf{v}_2)=\sum_{k=1}^d \sigma_k(V_1\overline{V}_2\phi_k)|_0,$ with $V_1$ and $V_2$ smooth sections of $\mathcal{V}$ extending $\mathbf{v}_1$ and $\mathbf{v}_2$ respectively. Upon setting $\Phi=\sum_{k=1}^d\sigma_k\phi_k$ we can suppose, in addition, that
\begin{equation}\label{prima}
\Phi=\sum_{j=1}^{p^{\prime\prime}}|z_j|^2-\sum_{j=p^{\prime\prime}+1}^{\nu}|z_j|^2+
\frac{1}{2}\sum_{l=1}^{p^\prime}t_l^2-
\frac{1}{2}\sum_{l=p^\prime+1}^{n-\nu}t_l^2+O(|s|(|z|+|s|+|t|)+|z|^3+|t|^3);
\end{equation}
see \cite[Section I.9]{t2} and \cite[(XVIII.3.2)]{t2} for details.
\end{remark}
 We can now prove Theorem \ref{teor2}. We may assume, without loss of generality, that
$\sigma=(1,0,\ldots,0)$. Consequently, from \eqref{prima} (after
the change of variables $t\mapsto t/\sqrt{2}$) we have
\begin{equation}
\phi_1(x,y,s,t)=|z^\prime|^2-|\zpp|^2+|\tp|^2-|\tpp|^2+O(|s|(|z|+|s|+|t|)+|z|^3+|t|^3),\label{fi1}
\end{equation}
where we set
\[\begin{cases}
\zp=(z_{1},\ldots,z_{p^{\prime\prime}}),\\ \zpp=(z_{p^{\prime\prime}+1},\ldots,z_{\nu})\\  \tp=(t_1,\ldots,t_{p^\prime}),\\ \tpp=(t_{p^\prime+1},\ldots, t_{n-\nu}).
\end{cases}
\]
Moreover, we choose a function $\chi(x,y,s,t)$ in $G^s_0(\R^{2\nu+d+(n-\nu)})$,  $\chi=0$ away from a neighborhood $V\subset\subset \Omega$ of $0$ and $\chi=1$ in a neighborhood $U\subset\subset V$ of $0$, where $V$ and $U$ will be chosen later on.  We set, for $\rho>0, \lambda>0$,
\[
\begin{split}
f_{\rho,\lambda}&=e^{\rho h_{1,\lambda}} d\ozp\wedge d\tp\\
v_{\rho,\lambda}&=\rho^{(m+n)/2}\chi{} e^{\rho
h_{2,\lambda}} d\ozpp\wedge d\tpp\wedge dz\wedge dw,
\end{split}\]
where, with $\lambda>1$,
\begin{equation}\label{h11}
h_{1,\lambda}:=-\iota s_1+\phi_1-2|\zp|^2-2|\tp|^2-\lambda\sum_{k=1}^d(s_k+\iota \phi_k)^2,
\end{equation}
and
\begin{equation}\label{h22}
h_{2,\lambda}:=\iota s_1-\phi_1-2|\zpp|^2-2|\tpp|^2-\lambda\sum_{k=1}^d(s_k+\iota\phi_k)^2.
\end{equation}
Now, we have $\chi\in G_0^s(\overline{V},C/2)$, for some $C>0$. We then apply Proposition \ref{prop2.1} with
$f_{\rho,\lambda}$
 and $v_{\rho,\lambda}$ in place of $f$ and $v$ respectively, and $C_1=C_2=C$, $\epsilon=C/2$,
  taking $V$ small enough to be contained in the neighborhood $\Omega'$ which arises in the conclusion of  Proposition \ref{prop2.1}. Observe that, in fact, $f_{\rho,\lambda}\in G^s
(\Omega,C;\Lambda^{0,q})$ since this form is in fact real-analytic and $p^{\prime\prime}+p^\prime=q$
by hypothesis, whereas $v_{\rho,\lambda}\in G^s_0(\overline{V},C/2;\Lambda^{m,n-q})\subset
G^s_0(\overline{\Omega'},C/2;\Lambda^{m,n-q})$ (recall, $m={\rm dim} M-n$). We prove now that $f_\rho$ is a cocycle (i.e. ${\rm d}^\prime f_{\rho,\lambda}=0$), so that Proposition \ref{prop2.1} can in fact be applied. However, we will show that \eqref{apriori} fails for every choice of $C'$
     when $\rho\to+\infty$, if $\lambda$ is large enough, obtaining a contradiction.\\
 Now, by
\eqref{ama2}
\[
{\rm d}^\prime f_{\rho,\lambda}=0,
\]
and
\begin{equation}\label{ag1}
{\rm d}^\prime v_{\rho,\lambda}=\rho^{(m+n)/2}e^{\rho
h_{2,\lambda}}{\rm d}^\prime \chi{}\wedge d\ozpp\wedge d\tpp\wedge dz\wedge dw.
\end{equation}
 In order to
estimate the right hand side of \eqref{apriori} we observe that, by
\eqref{fi1} and \eqref{h11}
\[
{\rm Re}\,
h_{1,\lambda}=-|\zp|^2-|\zpp|^2-|t|^2-\lambda|s|^2+\mathcal{R}(z,s,t)+
O(|z|^3+|t|^3)+\lambda O(|z|^4+|t|^4),\] where
\begin{equation}\label{gy}
|\mathcal{R}(z,s,t)|=O(|s|(|z|+|s|+|t|))\leq
\tilde{C}\left(\frac{\epsilon}{2}(|z|+|s|+|t|)^2+\frac{1}{2\epsilon}|s|^2\right),
\end{equation}
for every $\epsilon>0$. Hence, if $\epsilon$ and then  $1/\lambda$
are small enough we see that, possibly after replacing $\Omega$ with a smaller neighborhood,  \begin{equation}\label{ag2}
\sup_{\Omega}\,{\rm Re}\,h_{1,\lambda}\leq 0.
\end{equation}
Similarly,
\begin{multline*}
{\rm Re}\,h_{2,\lambda}=-|\zp|^2-|\zpp|^2-|t|^2-\lambda|s|^2\\
+\mathcal{R}^\prime(z,s,t)+O(|z|^3+|t|^3)+\lambda O(|z|^4+|t|^4),
\end{multline*}
with $\mathcal{R}^\prime$ satisfying the same estimate \eqref{gy}.
Therefore if $\lambda$ is sufficiently large, in $\Omega$ we
have
\[
{\rm Re}\, h_{2,\lambda}\leq-\frac{1}{2} (|z|^2+|t|^2+\lambda|s|^2)+\tilde{C}_1 (|z|^3+|t|^3)+\tilde{C}_2\lambda(|z|^4+|t|^4).
\]
Hence, possibly for a smaller $V$, since $U\subset\subset V$ is a neighborhood of $0$, there exists a constant $c>0$ such that
\begin{equation}\label{ag3}
\sup_{\overline{V}\setminus U}h_{2,\lambda}(z,s,t)\leq -c.
\end{equation}
As a consequence of Proposition \ref{prorodino} and \eqref{ag2}, \eqref{ag3}, for every compact subset $K\subset \Omega$ it turns out
\begin{equation}\label{aa}
\|f_{\rho,\lambda}\|_{K,C}\leq C^\prime e^{\rho^{1/s'}},
\end{equation}
\begin{equation}\label{ab}
\|{\rm d}^\prime v_{\rho,\lambda}\|_{K,C}\leq
C^{\prime\prime}\rho^{(m+n)/2} e^{-c\rho+\rho^{1/s'}},
\end{equation}
for any $1<s'<s$, where the constants $C^\prime, C^{\prime\prime}$ are independent of $\rho$. It follows that
\begin{equation}\label{3.3}
\|f_{\rho,\lambda}\|_{K,C}\|{\rm d}^\prime v_{\rho,\lambda}\|_{K,C}\leq
C^{\prime}C^{\prime\prime}\rho^{(m+n)/2}
e^{-c\rho+2\rho^{1/s'}}\longrightarrow 0\ {\rm as}\
\rho\to+\infty,
\end{equation}
because $s'>1$.
On the other hand, for the right-hand side of \eqref{apriori} it is easily seen that
 \[ \int f_{\rho,\lambda}\wedge
v_{\rho,\lambda}\longrightarrow c'\not=0,
\]
as $\rho\to+\infty$ (see the end of the proof of \cite[Theorem XVIII.3.1]{t2}), which together with \eqref{3.3} contradicts \eqref{apriori}.\par
This completes the proof of Theorem \ref{teor2}.

\begin{remark}\rm
The above machinery can be applied to prove other necessary conditions for Gevrey local solvability in the spirit of analogous results valid in the framework of smooth functions and Schwartz distributions.\par As an example, consider the special
case of local solvability
when $q=n$, namely in top
degree. In this case, the
Cordaro-Hounie condition
$(\mathcal{P})_{n-1}$ (see \cite{cordaro2} and
\cite{cor01}) is known to
be necessary for the local
solvability in the smooth category, and it is
conjectured to be sufficient
as well. Consider the following analytic variant. Let $L_j$,
$j=1,...,n$, be real-analytic independent
vector fields which generates
$\mathcal{V}$ at any point of
$\Omega$.\par
{\it We say that the real-analytic condition
 $(\mathcal{P})_{n-1}$ is
satisfied at $x_0$ if there
exists an open neighbourhood
$U\subset\Omega$ of $x_0$
such that, given any open set
$V\subset U$ and given any real-analytic
$h\in C^\infty(V)$ satisfying
$L_j h=0$, $j=1,...,n$, then
${\rm Re}\, h$ does not assume
a local minimum\footnote{A
real-valued function $f$
defined on a topological
space $X$ is said to assume a
local minimum over a compact
set $K\subset X$ if there
exist $a\in\R$ and $K\subset
V\subset X$ open such that
$f=a$ on $K$ and $f>a$ on
$V\setminus K$.} over any
nonempty compact subset of
$V$.}\par
One could then prove that
{\it if the real-analytic condition
$(\mathcal{P})_{n-1}$ is not
satisfied at $x_0$, then  for
every $s>1$, ${\rm d}'$ is
not locally solvable in the
sense of ultradistributions
of order $s$, near $x_0$ and
in degree $n$.}
\end{remark}
For the sake of brevity we omit the proof, which goes on along the same lines as that in \cite[Theorem 1.2]{cor01}, using the local solvability estimates in Proposition \ref{prop2.1}, combined with Proposition \ref{prorodino}.

\end{document}